\documentclass{article}

\usepackage{times}
\usepackage{graphicx} % more modern
\usepackage{subfigure} 

\usepackage{natbib}

\usepackage{algorithm}
\usepackage{algorithmic}

\usepackage{hyperref}

\usepackage{amsthm}
\usepackage{amsmath,graphicx}
\usepackage{amssymb}
\newtheorem{theorem}{Theorem}
\newtheorem{lemma}{Lemma}
\newtheorem{definition}{Definition}

\newtheorem{proposition}{Proposition}
\usepackage[accepted]{icml2017}

\icmltitlerunning{Gradient Projection Iterative Sketch for Large-Scale Constrained Least-Squares}

\begin{document} 

\twocolumn[
\icmltitle{Gradient Projection Iterative Sketch for Large-Scale Constrained Least-Squares}



% you can specify symbols, otherwise they are numbered in order

\begin{icmlauthorlist}
\icmlauthor{Junqi Tang}{ed}
\icmlauthor{Mohammad Golbabaee}{ed}
\icmlauthor{Mike E. Davies}{ed}
\end{icmlauthorlist}

\icmlaffiliation{ed}{Institute for Digital Communications, the University of Edinburgh, Edinburgh, UK}

\icmlcorrespondingauthor{Junqi Tang}{J.Tang@ed.ac.uk}

\icmlkeywords{Sketching, Large data optimization}

\vskip 0.3in
]



\printAffiliationsAndNotice{}  % leave blank if no need to mention equal contribution
%\footnotetext{hi}

\begin{abstract} 
% propose a simple randomized first order optimization algorithm Gradient Projection Iterative Sketch (GPIS) for efficiently solving large scale constrained Least Squares (LS). We provide a general theoretical analysis of the algorithm and demonstrate the mechanism by which it can reduce the computational cost through dimensional reduction. Our experimental results show that the GPIS can reach the same estimation accuracy as the standard first-order solver but with less computational cost.

We propose a randomized first order optimization algorithm Gradient Projection Iterative Sketch (GPIS) and an accelerated variant for efficiently solving large scale constrained Least Squares (LS). We provide the first theoretical convergence analysis for both algorithms. An efficient implementation using a tailored line-search scheme is also proposed. We demonstrate our methods' computational efficiency compared to the classical accelerated gradient method, and the variance-reduced stochastic gradient methods through numerical experiments in various large synthetic/real data sets.

\end{abstract} 

\section{Introduction}
We are now in an era of boosting knowledge and large data. In our daily life we have various signal processing and machine learning applications which involve the problem of tackling a huge amount of data. These applications vary from Empirical Risk Minimization (ERM) for statistical inference, to medical imaging such as the Computed Tomography (CT) and Magnetic Resonance Imaging (MRI), channel estimation and adaptive filtering in communications, and in machine learning problems where we need to train a neural network or a classifier from a large amount of data samples or images. Many of these applications involve solving constrained optimization problems. In a large data setting a desirable algorithm should be able to simultaneously address high accuracy of the solutions, small amount of computations and high speed data storage.

Recent advances in the field of randomized algorithms have provided us with powerful tools for reducing the computation for large scale optimizations. From the latest literature we can clearly see two streams of randomized algorithms, the first stream is the stochastic gradient descent (SGD) and its variance-reduced variants \cite{2013_Johnson_Accelerating}\cite{2013_Konecny_Semi-stochastic}\cite{defazio2014saga}\cite{allen2016katyusha}. The stochastic gradient techniques are based on the computationally cheap unbiased estimate of the true gradients with progressively reduced estimation variance. Although there has been several works on SGD techniques for performing constrained optimization \cite{xiao2014proximal}\cite{konevcny2016mini}, to the best of our knowledge, there are no results highlighting the computational speed up one could achieve by exploiting the data structure promoted by the constraint set. %as suggested in   \cite{chandrasekaran2013computational}\cite{agarwal2010fast}\cite{2015_Oymak_Sharp}.

This paper follows a second line of research and uses \emph{sketching} techniques, the crux of which is reducing the dimensionality of a large scale problem by random projections (e.g., sub-Gaussian matrices, Fast Johnson-Lindenstrauss Transforms (FJLT) \cite{2008_Ailon_Fast}\cite{2009_Ailon_Fast}, the Count Sketch \cite{clarkson2013low}, the Count-Gauss Sketch \cite{kapralov2016fake} or random sub-selection) so that the resulting sketched problem becomes computationally tractable. The meta-algorithms Classical Sketch (CS)\cite{mahoney2011randomized}\cite{drineas2011faster}\cite{2015_Pilanci_Randomized} and the Iterative Hessian Sketch (IHS) \cite{2016_Pilanci_Iterative} have been recently introduced for 
   solving efficiently large scale constrained LS problems which utilize the random sketching idea combined with the fact that solutions have low-dimensional structures such as sparsity in a properly-chosen dictionary, low-rank, etc. 
   
   %This paper proposes and analyzes the {\it Gradient Projection Iterative Sketch (GPIS)} based on equipping a combination of the meta-algorithms CS and IHS with the %powerful constraint-handling ability and 
  % PGD algorithm. As a first order solver, the PGD can provide us flexibility in handling the constraints and scalability in computation. The GPIS algorithm satisfies a number of our expectations as an efficient large scale constrained LS solver, which are mainly inherited from the scalability and flexibility of the PGD combined with dimensionality reducing properties of the sketching techniques: (a) computational efficiency,
% (b) efficiency on high speed storage, and
 %(c) flexibly to incorporate a wide range of constraints. 
 
 \subsection{Main Contributions}

\begin{itemize}
    \item {\bfseries Novel first order solvers based on iterative sketches for constrained Least-squares}
    
    We propose a basic first order algorithm Gradient Projection Iterative Sketch (GPIS) based on the combination of the {\it Classical Sketch} \cite{2015_Pilanci_Randomized} and {\it Iterative Hessian Sketch} \cite{2016_Pilanci_Iterative} for efficiently solving the constrained Least-squares, and also an accelerated variant by applying Nesterov's acceleration scheme \cite{nesterov2007gradient}\cite{nesterov2013gradient}.
    % the contribution towards IHS
    \item {\bfseries Theoretical analysis for both GPIS and Acc-GPIS}
    
    Although there exists established theories for the sketching programs in \cite{2015_Pilanci_Randomized}\cite{2016_Pilanci_Iterative} which describes their estimation performance under the assumption that the sketched programs are solved exactly,  there is no theoretical analysis of the use of first order methods within this framework, where each of the sketched programs are only approximately solved. The paper is the first one to provide this convergence analysis.
    
    \item {\bfseries Structure exploiting algorithms}
    
    In related theoretical works in sketching \cite{2015_Pilanci_Randomized}\cite{2016_Pilanci_Iterative}, convex relaxation \cite{chandrasekaran2013computational}, and the Projected Gradient Descent (PGD) analysis \cite{2015_Oymak_Sharp} with greedy step sizes when the data matrix is a Gaussian map, researchers have discovered that the constraint set is able to be exploited to accelerate computation. In this paper's convergence analysis of the proposed algorithms (which have an inner loop and an outer loop), we show explicitly how the outer loop's convergence speed is positively influenced by the constrained set. \footnote{Meanwhile we can show empirically that the inner loop is also being able to choose an aggressive step size with respect to the constraint. This extra step-size experiment can be found in the supplementary material.}

    \item {\bfseries Sketched gradients versus stochastic gradients -- quality versus quantity}
    
    The proposed GPIS algorithm draws a different line of research for first order randomized algorithms from the SGD and its recently introduced variance-reduced variants such as SVRG \cite{2013_Johnson_Accelerating} and SAGA \cite{defazio2014saga} by utilizing randomized sketching techniques and deterministic iterations instead of the stochastic iterations. This approach leads to convenience in optimally choosing the step size by implementing line search because it follows the classical results and techniques in first order optimization. Although such stochastic gradient algorithms have good performance in terms of epoch counts when a small minibatch size is used, this type of measure does not consider at least three important aspects: 1) the computational cost of projection / proximal operator, 2) the modern computational devices are usually more suitable for vectorized / parallel computation,
    3) the operational efforts to access new data batches each iteration (note that the large data should be stored in large memories, which are usually slow). 
    
    It is well known that the small batch size in stochastic gradients usually leads to a greater demand on the number of iterations. In the cases where the  projection / proximal operator is costly to compute, for instance, if we wish to enforce sparsity in a transformed domain, or an analytical domain (total-variation), we would need to use a large batch size in order to control computation which generally would not be favorable for stochastic gradients techniques as they usually achieves best performance when small batch size is used. In this paper we have designed experiments to show the time efficiency of the sketched gradients with Count-sketch \cite{clarkson2013low} and an aggressive line-search scheme for near-optimal choice of step size each iteration \cite{nesterov2007gradient} compared to a mini-batched version of the SAGA algorithm \cite{defazio2014saga} and the accelerated full gradient method \cite{2009_Beck_Fast} in large scale constrained least-square problems.

\end{itemize}

\subsection{Background}
Consider a constrained Least-squares regression problem in the large data setting. We have the training data matrix $A \in \mathbb{R}^{n \times d}$ with $n > d$ and observation $y \in \mathbb{R}^n$. Meanwhile we restrict our regression parameter to a convex constrained set  $\mathcal{K}$ to enforce some desired structure such as sparsity and low-rank\footnote{In scenarios where we do not know the exact constraint $\mathcal{K}$, we may wish to use {\it regularized} least-squares instead of strict constraint. This paper focus on the constrained case and leave the extension for the proximal setting as future work.}:
	\begin{equation}\label{LS}
    x^\star=\arg\min_{x\in \mathcal{K}} \left\{f(x) := \|y-Ax\|^2_2 \right\}.
    \end{equation}
Then we define the error vector $e$ as:
\begin{equation}\label{eq:3}
   e = y - Ax^\star
\end{equation}
A standard first order solver for (\ref{LS}) is the projected gradient algorithm (we denote the orthogonal projection operator onto the constrained set $\mathcal{K}$ as $\mathcal{P_K}$):
\begin{equation} \label{eq:4}
    x_{j+1}=\mathcal{P}_\mathcal{K}(x_{j}-\eta A^T(Ax_{j}-y)).
\end{equation}
Throughout the past decade researchers proposed a basic meta-algorithm for approximately solving the Least-squares problem that we call the {\it Classical Sketch} (CS), see e.g. \cite{mahoney2011randomized} \cite{drineas2011faster} \cite{2015_Pilanci_Randomized}, which compresses the dimension of the LS and makes it cheaper to solve.  The Johnson-Lindenstrauss theory \cite{johnson1984extensions} \cite{dasgupta2003elementary} and the related topic of Compressed Sensing \cite{2006_Donoho_Compressed}\cite{candes2006stable}\cite{baraniuk2008simple} revealed that random projections can achieve stable embeddings of  
high dimensional data into lower dimensions and that the number of measurements required is proportional to the intrinsic dimensionality of data (as opposed to the ambient dimension) which is manifested in the set of constraints $\mathcal{K}$. This motivates replacing the original constrained LS problem with a sketched LS \cite{2015_Pilanci_Randomized}:

	\begin{equation}\label{CS}
    \hat{x}=\arg\min_{x\in \mathcal{K}}\left\{f_0(x) := \|Sy-SAx\|^2_2\right\},
    \end{equation}
where the sketching matrix $S \in \mathbb{R}^{m \times n}, m \ll n$ is a random projection operator which satisfies:

\begin{equation}\label{eq:6}
    E\left(\frac{S^TS}{m}\right)=I.
\end{equation}
When the embedding dimension $m$ is larger than a certain factor of the true solution's intrinsic dimension (measured through a statistical tool called the Gaussian Width \cite{2012_Chandrasekaran_Convex}), the {\it Classical Sketch} (\ref{CS}) ensures a robust estimation of $x^\star$ with a noise amplification factor compared to the estimator given by solving the original LS problem (\ref{LS}), and it has been shown that the smaller the embedding dimension $m$ is, the bigger the noise amplification factor will be. To get a sketching scheme for the scenarios where a high accuracy estimation is demanded, a new type of meta-algorithm {\it Iterative Hessian Sketch} (IHS) was introduced by Pilanci and Wainwright \cite{2016_Pilanci_Iterative}:
\begin{equation}\label{IHS}
\begin{aligned}
       x^{t+1}= \arg \min_{x \in \mathcal{K}} \{f_{t}(x) := &&\frac{1}{2m}\|S^{t}A(x-x^t)\|_2^2\\ &&-x^TA^T(y-Ax^t)\}.
\end{aligned}
\end{equation}
At the $t$th iteration of IHS a new sketch of the data matrix $S^{t}A$ and a full gradient $A^T(y-Ax^t)$ at the current estimate $x^t$ is calculated to form a new sketched least-square problem. By repeating this procedure the IHS will converge to the solution of the original problem (\ref{LS}) in  typically a small number of iterations. The iterative sketch essentially
corrects the noise amplification and enables $(1 + \epsilon)$ LS
accuracy in the order of $\log \frac{1}{\epsilon}$ outer loop iterations.

\section{Gradient Projection Iterative Sketch}

 \subsection{The Proposed Algorithms}
Here we consider the combination of CS with the first order PGD algorithm, the {\it Gradient Projection Classical Sketch} (GPCS):
\begin{equation} \label{eq:7}
    x_{i+1}=\mathcal{P}_\mathcal{K}(x_{i}-\eta (S^0A)^T(S^0Ax_{i}-S^0y)).
\end{equation}
Similarly we obtain the {\it Gradient Projection Iterative Hessian Sketch} (GPIHS) for solving IHS (\ref{IHS}):
\begin{equation} \label{eq:8}
    x_{i+1}=\mathcal{P}_\mathcal{K}(x_i-\eta((S^tA)^T(S^tA)(x_i-x^t)+mA^T(Ax^t-y)).
\end{equation}
Our proposed GPIS algorithm applies PGD to solve a sequence of sketched LS, starting with a CS step for a fast initialization, and then is followed by further iterations of IHS. We can observe from Algorithm 1 that sketches are constructed in the outer loop and within the inner loop we only need to access them. This property could be very useful when, for instance $A$ is stored in a slow speed memory and it is too large to be loaded at once into the fast memory, or in large scale image reconstruction problems such as CT where due to its prohibited size $A$ is constructed on the fly.
 Note that thanks to the sketching each inner iteration of GPIS is $\frac{n}{m}$ times cheaper than a full PGD iterate in terms of matrix-vector multiplication, so intuitively we can see that there is potential in Algorithm 1 to get computational gain over the standard first order solver PGD.
 
 \begin{algorithm}[tb]\label{A1}
   \caption{Gradient Projection Iterative Sketch --- $\mathcal{G}(m, [\eta], [k])$}
\begin{algorithmic}
   \STATE  Initialization: $x_0^0 = 0$
    \STATE Given $A \in \mathbb{R}^{n \times d}$, sketch size $m \ll n$
    \STATE Prior knowledge: the true solution $x$ belongs to set $\mathcal{K}$ 

    \STATE Run GPCS iterates (Optional): 
    \STATE Generate a random sketching matrix $S^0 \in \mathcal{{R}}^{m \times n}$
    \STATE Calculate $S^0A$, $S^0y$ 
   \FOR{$i=1$ {\bfseries to} $k_0$}
   \STATE        $x_{i+1}^0=\mathcal{P}_\mathcal{K}(x_{i}^0-\eta_{0,i} (S^0A)^T(S^0Ax_{i}^0-S^0y))$
   
   \ENDFOR
\STATE $x_0^1 = x_{k_0}^0$

\STATE Run GPIHS iterates

\FOR{$t=1$ {\bfseries to} $N$ }
     \STATE Calculate $g=A^T(Ax_0^t-y)$
     \STATE  Generate a random sketching matrix $S^t \in \mathcal{{R}}^{m \times n}$
     \STATE Calculate $A^t_s=S^tA$ 

   \FOR{$i=1$ {\bfseries to} $k_t$}
   \STATE               $ x_{i+1}^t=\mathcal{P}_\mathcal{K}(x_i^t-\eta_{t,i}(A_s^{t^T}A_s^t(x_i^t-x_0^t)+mg))$
   
   \ENDFOR

\STATE  $x_0^{t+1}=x_{k_t}^t$

\ENDFOR
\end{algorithmic}
\end{algorithm}

Since it is well-known that in convex optimization the standard first order method Projected/proximal gradient descent can be accelerated by Nesterov's acceleration scheme \cite{nesterov2007gradient} \cite{nesterov2013gradient} \cite{2009_Beck_Fast}, our Algorithm 1 has potential to be further improved by introducing Nesterov's acceleration. Here we propose Algorithm 2 -- {\it Accelerated Gradient Projection Iterative Sketch} (Acc-GPIS) which is based on the combination of the accelerated PGD and iterative sketching.

 \begin{algorithm}[tb]\label{A2}
   \caption{Accelerated Gradient Projection Iterative Sketch --- $\mathcal{A}(m, [\eta], [k])$}
\begin{algorithmic}

   \STATE  Initialization: $x_0^0 = 0$, $\tau_0 = 1$
    \STATE Given $A \in \mathbb{R}^{n \times d}$, sketch size $m \ll n$
    \STATE Prior knowledge: the true solution $x$ belongs to set $\mathcal{K}$ 

    \STATE Run GPCS iterates (Optional): 
    \STATE Generate a random sketching matrix $S^0 \in \mathcal{{R}}^{m \times n}$
    \STATE Calculate $S^0A$, $S^0y$ 
   \FOR{$i=1$ {\bfseries to} $k_0$}
   \STATE        $x_{i+1}^0=\mathcal{P}_\mathcal{K}(z_{i}^0-\eta_{0,i} (S^0A)^T(S^0Az_{i}^0-S^0y))$
%   \STATE       find $p_{i+1}^0$ solves $(p_{i+1}^0)^2 = (1 - p_{i+1}^0)(p_{i}^0)^2$
%   \STATE $\tau_{i+1}^0 = p_{i}^0(1-p_{i}^0)/[(p_{i}^0)^2 + p_{i+1}^0]$
    \STATE $\tau_i = (1 + \sqrt{1 + 4\tau_{i-1}^2})/2$ 
    \STATE {\bfseries Extrapolate} $z_{i+1}^0 = x_{i+1}^0 + \frac{\tau_{i-1} - 1}{\tau_i}(x_{i+1}^0 - x_{i}^0)$

   \ENDFOR
\STATE $x_0^1 = z_0^1 = x_{k_0}^0$

\STATE Run GPIHS iterates

\FOR{$t=1$ {\bfseries to} $N$ }
     \STATE Calculate $g=A^T(Ax_0^t-y)$
     \STATE  Generate a random sketching matrix $S^t \in \mathcal{{R}}^{m \times n}$
     \STATE Calculate $A^t_s=S^tA$ 
     \STATE $\tau_0 = 1$
   \FOR{$i=1$ {\bfseries to} $k_t$}
   \STATE               $x_{i+1}^t=\mathcal{P}_\mathcal{K}(z_i^t-\eta_{t,i}(A_s^{t^T}A_s^t(z_i^t-x_0^t)+mg))$
 %  \STATE       find $p_{i+1}^t$ solves $(p_{i+1}^t)^2 = (1 - p_{i+1}^t)(p_{i}^t)^2$
  % \STATE $\tau_{i+1}^t = p_{i}^t(1-p_{i}^t)/[(p_{i}^t)^2 + p_{i+1}^t]$
    \STATE $\tau_i = (1 + \sqrt{1 + 4\tau_{i-1}^2})/2$ 
    \STATE {\bfseries Extrapolate}  $z_{i+1}^t = x_{i+1}^t + \frac{\tau_{i-1} - 1}{\tau_i}(x_{i+1}^t - x_{i}^t)$
   \ENDFOR
\STATE $x_0^{t+1}= z_0^{t+1} =x_{k_t}^t$
\ENDFOR
\end{algorithmic}
\end{algorithm}

One of the benefits of deterministically minimising the sketched cost function can bring is that the implementation of the line-search scheme can be easy and provably reliable since the underlying sketched cost function each outer loop is fixed. For example \cite{nesterov2007gradient} provides a simple line-search scheme for gradient methods to make the step size of each iteration to be nearly optimal, with rigorous convergence theory and also a explicit bound for the number of additional gradient calls. The line-search scheme is described by Algorithm 3. On the other hand in the stochastic gradient literature there are no practical strategies for efficient line search in the case of constrained optimization. To the best of our knowledge, only the SAG paper \cite{schmidt2013minimizing} addresses the issue of line-search but their implementation is only for unconstrained optimization.

 \begin{algorithm}[tb]
   \caption{line-search scheme for GPIS and Acc-GPIS --- $\mathcal{L}(x_i, f_t(x), \triangledown f_t(x_i), \gamma_u, \gamma_d )$ \cite{nesterov2007gradient}}
\begin{algorithmic}

\STATE {\bfseries Input:} update $x_i$, sketched objective function $f_t(x)$, gradient vector $\triangledown f_t(x_i)$, line search parameters $\gamma_u$ and $\gamma_d$, step size of previous iteration $\eta_{i-1}$.
\STATE Define composite gradient map $m_L$:
\STATE $m_L := f_t(x_i) + (x - x_i)^T \triangledown f_t(x_i) + \frac{1}{2\eta}\|x - x_i\|_2^2$
\STATE $\eta = \gamma_d \eta_{i-1}$  
\STATE $x = \mathcal{P_K}(x_i - \eta \triangledown f_t(x_i))$

\WHILE{ $f_t(x) \geq m_L $ }

\STATE $\eta = \eta/\gamma_u$  
\STATE $x = \mathcal{P_K}(x_i - \eta \triangledown f_t(x_i))$

\ENDWHILE 

\STATE {\bfseries Return} $x_{i+1} = x$ and $\eta_i = \eta$
\end{algorithmic}
\end{algorithm}

\section{Convergence Analysis}

\subsection{General Theory}

We start our theoretical analysis by some definitions:
\begin{definition}
The Lipschitz constant $L$ and strong convexity $\mu$ for the LS (\ref{LS}) are defined as the largest and smallest singular values of the Hessian matrix $A^TA$:
\begin{equation}
    \mu \|z_d\|_2^2\leq \|Az_d\|_2^2 \leq L \|z_d\|_2^2,
\end{equation}
 for all $z_d \in \mathbb{R}^d$, where $0 \leq \mu < L$ ($\mu = 0$ means the LS (\ref{LS}) is non-strongly convex).
\end{definition}
\begin{definition}
Let $\mathcal{C}$ be the smallest closed cone at $x^\star$ containing the set $\mathcal{K}-x^\star$:
\begin{equation}
    \mathcal{C} = \left\{p \in \mathbb{R}^d |\  p = c(x - x^\star) , \forall c \geq 0, x \in \mathcal{K} \right\},
\end{equation}
$\mathbb{S}^{d-1}$ be the unit sphere in $\mathbb{R}^d$,
$\mathcal{B}^{d}$ be the unit ball in $\mathbb{R}^d$, $z$ be arbitrary fixed unit-norm vectors in $\mathbb{R}^n$. The factors $\alpha(\eta,S^tA)$,   $\rho(S^t,A)$ and $\sigma(S^t,A)$ are defined as:
\begin{align}
&\alpha(\eta_t,S^tA)=\sup_{u,v \in \mathcal{B}^d} v^T(I-\eta_t A^TS^{t^T}S^tA)u,\\
&\rho(S^t,A)=\frac{\sup_{v \in A\mathcal{C} \cap \mathbb{S}^{n-1}} v^T ( \frac{1}{m}S^{t^T}S^t-I)z}{\inf_{v \in A\mathcal{C} \cap \mathbb{S}^{n-1}} \frac{1}{m}\|S^tv\|_2^2},\\
&\sigma(S^t, A)= \frac{\sup_{v \in range(A) \cap \mathbb{S}^{n-1}} \|S^t v\|_2^2}{\inf_{v \in  range(A) \cap \mathbb{S}^{n-1}} \|S^t v\|_2^2},
\end{align}
\end{definition}

For convenience, we denote each of this terms as: $\alpha_t := \alpha(\eta_t,S^tA)$, $\rho_t := \rho(S^t,A)$ and $\sigma_t := \sigma(S^t, A)$. Our theory hangs on these three factors and we will show that they can be bounded with exponentially high probabilities for Gaussian projections.
\begin{definition}
The optimal points $x_\star^t$ of the sketch programs $f_t(x)$ are defined as:
\begin{equation}
    x_\star^t = \arg \min_{x \in \mathcal{K}} f_t(x).
\end{equation}
We also define a constant R for the simplicity of the theorems:
\begin{equation}
    R = \max_t \|x_\star^t - x^\star\|_2^2.
\end{equation}
\end{definition}
We use the notation $\|v\|_A = \|Av\|_2$ to describe the $A$-norm of a vector $v$ in our theory. After defining these properties we can derive our first theorem for GPIS when $f(x)$ is strongly convex, e.g, $\mu > 0$ :
\begin{theorem}
(Linear convergence of GPIS when $\mu > 0$) For fixed step sizes $\eta_t \leq \frac{1}{\|S^tA\|_2^2}$, the following bounds hold:
for $t = 0$ (the initialization loop by GPCS),
\begin{equation}
       \|x_0^1 - x^\star\|_A \leq (\alpha_0)^{k_0} \sqrt{\frac{L}{\mu}} \| x_0^0 - x_\star^0 \|_A + 2 \rho_0 \|e\|_2,
\end{equation}
for $N \geq 1$ and $x^t_0:=x^{t-1}_{k_{t-1}}$  
(the consecutive loops by GPIHS),
    \begin{equation}
       \|x_0^{N + 1} - x^\star\|_A 
       \leq  \left\{\prod_{t=1}^N\rho_t^\star\right\}\|x_0^1-x^\star\|_A;
    \end{equation}
where we denote:
\begin{equation}
    \rho_t^\star = (\alpha_t)^{k_t}\left[(1 + \rho_t)\sqrt{\frac{L}{\mu}}\right] + \rho_t
\end{equation}
\end{theorem}

From Theorem 1 we can see that when we have strong convexity, aka $\mu > 0$, by choosing a appropriate step size the GPCS loop will linearly converge to a sub-optimal solution, the accuracy of which depends on the value of $2 \rho_0 \|e\|_2$; and the following GPIHS iterations enjoys a linear convergence towards the optimal point.

When the least-squares solution is relatively consistent ($\|e\|_2$ is small), the GPCS loop will provide excellent initial convergence speed, otherwise it is not beneficial -- that's why we say that the GPCS loop is optional for our GPIS / Acc-GPIS algorithm. For regression problems on data sets, we advise not to run the GPCS iterates, but for signal/image processing applications, we would recommend it.

For the cases where the strong convexity is not guaranteed ($\mu \geq 0$) we show the $\mathcal{O}(\frac{1}{k})$ convergence rate for GPIS algorithm:

\begin{theorem}
(Convergence guarantee for GPIS when $\mu \geq 0$) If we choose a fixed number ($k$) of inner-loops for $t = 1,...,N$, the following bounds hold: for $t = 0$,
\begin{equation}
       \|x_0^1 - x^\star\|_A \leq \sqrt{\frac{\beta L\sigma_0 R}{2k_0}} + 2 \rho_0 \|e\|_2,
\end{equation}
for $N \geq 1$ and $x^t_0:=x^{t-1}_{k}$  
\begin{equation}
\begin{aligned}
       \|x_0^{N + 1} - x^\star\|_A &\leq \left\{\prod_{t = 1}^N \rho_t \right\}\| x_0^1 - x^\star \|_A\\
       &+ \frac{\max_t \sqrt{\sigma_t}}{1 - \max_t \rho_t} \sqrt{\frac{\beta LR}{2k}},    
\end{aligned}
\end{equation}   
where $\beta = 1$ for fixed step sizes $\eta_t = \frac{1}{\|S^tA\|_2^2}$, $\beta = \gamma_u$ for a line search scheme described by Algorithm 3 with parameter $\gamma_u > 1$ and $\gamma_d = 1$.
\end{theorem}

For the Accelerated GPIS algorithm we also prove the desired $\mathcal{O}(\frac{1}{k^2})$ convergence rate:

\begin{theorem}
(Convergence guarantee for Accelerated GPIS when $\mu \geq 0$) If we choose a fixed number ($k$) of inner-loops for $t = 1,...,N$, the following bounds hold:
for $t = 0$ ,
\begin{equation}
       \|x_0^1 - x^\star\|_A \leq \sqrt{\frac{2\beta L\sigma_0 R}{(k_0 + 1)^2}} + 2 \rho_0 \|e\|_2,
\end{equation}
for $N \geq 1$ and $x^t_0:=x^{t-1}_{k}$  
\begin{equation}
\begin{aligned}
       \|x_0^{N + 1} - x^\star\|_A &\leq \left\{\prod_{t = 1}^N \rho_t \right\}\| x_0^1 - x^\star \|_A\\
       &+ \frac{\max_t \sqrt{\sigma_t}}{1 - \max_t \rho_t} \sqrt{\frac{2\beta LR}{(k + 1)^2}},
\end{aligned}
\end{equation}   
where $\beta = 1$ for fixed step sizes $\eta_t = \frac{1}{\|S^tA\|_2^2}$, $\beta = \gamma_u$ for a line search scheme described by Algorithm 3 with parameter $\gamma_u > 1$ and $\gamma_d = 1$.
\end{theorem}

We include the proofs in our supplementary material. 
It is well known that for the case $\mu > 0$, the accelerated gradients can potentially enjoy the improved linear rate $\mathcal{O}((1 - \sqrt{\frac{\mu}{L}}))$ but it demands the exact knowledge of the value $\mu$ (which is often unavailable in practical setups). In our implementation for the Acc-GPIS method in the experiments, we use the adaptive {\it gradient restart} scheme proposed by \cite{o2015adaptive}.

\subsection{Explicit Bounds for Gaussian Sketches}

The theorems above provide us with a framework to describe the convergence of GPIS and Acc-GPIS in terms of the constants $\alpha$, $\rho$ and $\sigma$. For Gaussian sketches, these constants find explicit bounding expressions in terms of the sketch size $m$ and the complexity of the constraint cone $\mathcal{C}$. 
%We have expressions of these terms to the Gaussian Width of the set $\mathcal{W}(A\mathcal{C}\cap\mathcal{S}^{n-1})$ and the sketch dimension $m$ when we choose $S^t$ to be a Gaussian sketching matrix. 
For this, we use the Gaussian Width argument (see, e.g. \cite{2012_Chandrasekaran_Convex}):

\begin{definition} 
The Gaussian width $\mathcal{W}(\Omega)$ is a statistical measure of the size of a set $\Omega$:
\begin{align}
    \mathcal{W}(\Omega) = E_g\left( \sup_{v \in \Omega} v^Tg\right),
\end{align}
where $g \in \mathbb{R}^n$ is draw from i.i.d. normal distribution.
\end{definition}
The value of $\mathcal{W}(\mathcal{C}\cap\mathbb{S}^{d-1})$ is an useful measure of the tightness of the structure of $x^\star$. For example, if $x^\star$ is $s$-sparse and we model the sparsity constraint using an $l_1$ ball, we will have $\mathcal{W}(\mathcal{C}\cap\mathbb{S}^{d-1}) \leq \sqrt{2slog(\frac{d}{s}) + \frac{5}{4}s}$, which means the sparser $x^\star$ is, the smaller the $\mathcal{W}(\mathcal{C}\cap\mathbb{S}^{d-1})$ will be \cite{2012_Chandrasekaran_Convex}. As an illustration we now quantify the bounds in our general theorems in terms of the sketch size $m$ and the Gaussian width of the transformed cone  $\mathcal{W}(A\mathcal{C}\cap\mathbb{S}^{n-1}) \leq \sqrt{d}$, and the ambient dimension of the solution domain ($d$). Now we are ready to provide the explicit bounds for the factors $\alpha_t$, $\rho_t$ and $\sigma_t$ for the general theorems (we denotes $b_m :=\sqrt{2}\frac{\Gamma(\frac{m+1}{2})}{\Gamma(\frac{m}{2})}\approx\sqrt{m}$ \cite{2015_Oymak_Sharp} and $\mathcal{W} := \mathcal{W}(A\mathcal{C}\cap\mathbb{S}^{n-1})$ for the following lemmas):

\begin{proposition}

If the step-size $\eta_t = \frac{1}{L(b_m + \sqrt{d} + \theta)^2}$, sketch size $m$ satisfies $b_m > \sqrt{d}$, and  the entries of the sketching matrix $S^t$ are i.i.d drawn from Normal distribution, then:
\begin{equation}
    \alpha_t \leq \left\{1-\frac{\mu}{L}\frac{(b_m-\sqrt{d}- \theta)^2}{(b_m+\sqrt{d}+ \theta)^2}\right\},
\end{equation}
with probability at least $(1-2e^{-\frac{\theta^2}{2}})$.
\end{proposition}
\begin{proposition}
If the entries of the sketching matrix $S^t$ are i.i.d drawn from Normal distribution, then:
    \begin{equation}
        \rho_t \leq \frac{m}{(b_m-\mathcal{W}- \theta)^2} \left(\frac{\sqrt{2}b_m(\mathcal{W} + \theta)}{m} + |\frac{b_m^2}{m} - 1| \right),
    \end{equation}
    With probability at least $(1-e^{-\frac{\theta^2}{2}})(1-8e^{-\frac{\theta^2}{8}})$.
\end{proposition}

\begin{proposition}
If the entries of the sketching matrix $S^t$ are i.i.d drawn from Normal distribution, and the sketch size $m$ satisfies $b_m > \sqrt{d}$, then:

\begin{equation}
    \sigma_t \leq \frac{(b_m + \sqrt{d} + \theta)^2}{(b_m - \sqrt{d} - \theta)^2} 
\end{equation}
with probability at least $(1-2e^{-\frac{\theta^2}{2}})$.
\end{proposition}

(We include the proofs in the supplementary material.) We would like to point out that our bound on factor $\rho_t$ in proposition 2 has revealed that the outer-loop convergence of GPIS and Acc-GPIS relies on the Gaussian Width of the solution $x^\star$ and the choice of the sketch size $m$:
\begin{equation}
    \rho_t \lesssim \frac{\sqrt{2}\frac{\mathcal{W}}{\sqrt{m}}}{(1 - \frac{\mathcal{W}}{\sqrt{m}})^2}.
\end{equation}
We can then observe that the larger the sketch size $m$ is with respect to $\mathcal{W}$, the faster the outer loop convergence of GPIS and Acc-GPIS can be, but on the other hand we should not choose $m$ too large otherwise the inner-loop iteration become more costly -- this trade-off means that there is always a sweet spot for the choice of $m$ to optimize the computation.

Our theory is conservative in a sense that it does not provide guarantee for a sketch size which is below the ambient dimension $d$ since the factors $\alpha_t$ and $\sigma_t$ which are related to the inner loop prohibit this. %Numerically, for large data regressions ($n \gg d$) we are interested in, we observe that a choice of $m \asymp \mathcal{O}(d)$ typically provides the best overall performance.

    Although the Gaussian sketch provides us strong guarantees, due to computational cost of dense matrix multiplication, which is of $\mathcal{O}(mnd)$, it is not computationally attractive in practice. In the literature of randomized numerical linear algebra and matrix sketching, people usually use the random projections with fast computational structures such as the Fast Johnson-Lindenstrauss Transform \cite{2008_Ailon_Fast}\cite{2009_Ailon_Fast}, Count sketch \cite{clarkson2013low} and Count-Gauss sketch\cite{kapralov2016fake}, which cost $\mathcal{O}(nd\log(d))$, $\mathcal{O}(nnz(A))$ and $\mathcal{O}(nnz(A) + m^{1.5}d^3)$ respectively. These fast sketching methods provide significant speed up in practice compared to Gaussian sketch when $n \gg d$.

\section{Implementation for GPIS and Acc-GPIS in Practice}

In this section we describe our implementation of GPIS and Acc-GPIS algorithm in the experiments:

\begin{itemize}
    \item {\bfseries Count sketch}
 In this paper we choose the Count Sketch as our sketching method since it can be calculated in a streaming fashion and we observe that this sketching method provides the best computational speed in practice. A MATLAB implementation for efficiently applying the Count Sketch can be found in \cite{wang2015practical}.
    \item {\bfseries Line search}
    We implement the line-search scheme given by \cite{nesterov2007gradient} and is described by Algorithm 3 for GPIS and Acc-GPIS in our experiments with parameters $\gamma_u = 2$, and $\gamma_d = 2$. 
    \item {\bfseries Gradient restart for Acc-GPIS}
    We choose a efficient restarting scheme {\it gradient restart} proposed by \cite{o2015adaptive}.
    
   % \cite{o2015adaptive} has proposed two heuristic adaptive restart schemes - {\it gradient restart} and {\it function restart} for the accelerated gradient methods and have shown significant improvements without the need of the knowledge of the functional parameters $\mu$ and $L$. Such restart methods are directly applicable for the Acc-GPIS by nature due to its sketched deterministic iterations. Here we choose the {\it gradient restart} since it achieves comparable performance in practice as {\it function restart} but cost only $\mathcal{O}(d)$ operations.
    
\end{itemize}

\section{Numerical Experiments}

\subsection{Settings for Environments and Algorithms}

We run all the numerical experiments on a DELL laptop with 2.60 GHz Intel Core i7-5600U CPU and 1.6 GB RAM, MATLAB version R2015b.

We choose two recognized algorithms to represent the the full gradients methods and the (incremental) stochastic gradient method. For the full gradient, we choose the Accelerated projected gradient descent \cite{2009_Beck_Fast} \cite{nesterov2013introductory} with line-search method described in Algorithm 3 and gradient restart to optimize its performance. For the stochastic gradients we choose a mini-batched version of SAGA \cite{defazio2014saga} with various batch sizes ($b = 10$, $b = 50$ and $b = 100$). We use the step size suggested by SAGA's theory which is $\frac{1}{3\hat{L}}$. The code for the minibatch SAGA implementation can be found in (https://github.com/mdeff/saga). We get the estimated value for $\hat{L}$ by averaging the largest singular value of each batch (note that we do not count this into the elapsed time and epoch counts for SAGA). The sketch size of our proposed methods for each experiments are list in Table 1. We use the $l_1$ projection operator provided by the SPGL1 toolbox \cite{van2007spgl1} in the experiments.
%\footnote{There is a generic acceleration scheme {\it Catalyst} \cite{lin2015universal} which can potentially accelerate SAGA, but its parameter is difficult to be chosen in practice, as discussed in \cite{allen2016katyusha}, hence we do not include this into our experiments}

\begin{table}[h]
\caption{Sketch sizes ($m$) for GPIS and Acc-GPIS for each experiments}
\label{sample-table}
\vskip 0.15in
\begin{center}
\begin{small}
\begin{sc}
\begin{tabular}{lcccr}
\hline
\abovespace\belowspace
 syn1 &syn2 &syn3& MAGIC04  & YEAR \\
\hline
\abovespace\belowspace
800 & 800 & 400  & 475  & 1000\\
\hline
\end{tabular}
\end{sc}
\end{small}
\end{center}
\vskip -0.1in
\end{table}

%In the large data regression it is better for our GPIS and Acc-GPIS algorithms not to start with a GPCS iteration but directly go to the GPIHS iterations, but for signal processing and image reconstruction applications we would recommend it.

%The number of inner loops for both GPIS and Acc-GPIS are fixed numbers appropriately tuned for each experiments. 

\subsection{Synthetic Data Sets}

We start with some numerical experiments on synthetic problems (Table 2) to gain some insights into the algorithms. We begin by focusing on $l_1$ norm constrained problems. We generate synthetic constrained least-square problems by first generating a random matrix sized $n$ by $d$, then perform SVD on such matrix and replace the singular values with a logarithmically decaying sequence. (The details of the procedure can be found in supplementary materials.)
Similarly we generate a synthetic problem (Syn3) for low-rank recovery with nuclear-norm constraint. This is also called the {\it multiple response regression} with a generalized form of the Least-squares:
\begin{equation}\label{Q}
    X^\star  = \arg \min_{\|X\|_\star \leq r} \|| Y  - AX |\|_F^2.
\end{equation}
\subsection{Real Data Sets}
We first run an unconstrained least-squares regression on the Year-prediction (Million-song) data set from UCI Machine Learning Repository \cite{Lichman:2013} after we normalize each column of the data matrix. We use this example to demonstrate our algorithms' performance in unconstrained problems.

Then we choose Magic04 Gamma Telescope data set from \cite{Lichman:2013} to generate a constrained Least-square regression problem. The original number of features for Magic04 are 10 , and we normalize each columns of the original data matrix and additional irrelevant random features as the same way as the experiments in \cite{langford2009sparse}\cite{shalev2011stochastic} to the data sets so that the regressor $x^\star$ can be chosen to select the sparse set of relevant features by again solving (\ref{LS}). For this case we first precalculate the $l_1$-norm of the original program's solution and then set it as the radius of our $l_1$ constraint. The details of the real data sets can be found in Table 3.

	\begin{figure}[t] %\label{fig:3} %{\textwidth}
	\centering

		\includegraphics[width=9cm]{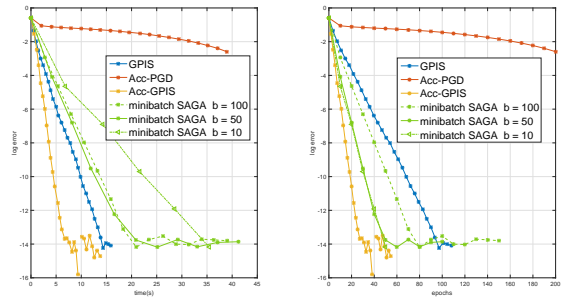}
	\caption{Experimental results on Million-song Year prediction data set (unconstrained LS regression experiment)}
	\end{figure} %\vspace{-1cm} 

\subsection{Discussion}
%The incremental gradient methods like SAGA and SVRG are powerful in the sense that they can choose arbitrary small mini-batch sizes and get efficiency in terms of epoch counts which are widely used in stochastic optimization literature to measure an algorithm's computational complexity (1 epoch's computational cost = 1 evaluation of the full gradient).

We measure the performance of the algorithms by the wall-clock time (simply using the tic toc function in MATLAB) and the epoch counts. The $y$-axis of each plot is the relative error $\log(\frac{f(x) - f(x^\star)}{f(x^\star)})$. The values below $10^{-10}$ are reported as exact recovery of the least-square solution.

In all the experiments, our methods achieve the best performance in terms of wall-clock time. We show that in many cases the sketched gradient methods can outperform leading stochastic gradient methods. Both sketched gradients and stochastic gradients can achieve reduced complexity compared to the (accelerated) full gradient method, but since the sketched method has inner-loops with deterministic iterations, the line-search scheme of the classic gradient descent method can be directly used to make each iteration's step size be near optimal, and unlike the stochastic gradient, our methods do not need to access new mini-batches from memory each iteration, which can save operational time in practice.

SAGA performs competitively in terms of epoch counts (right hand figures) which is generally achieved using a small batch size of 10. Unfortunately the additional cost of the projection per iteration can severely impact on the wall clock time performance\footnote{For the unconstrained case (Million-song data set, sized $5\times10^5$ by 90), we also observe that, SAGA with $b = 10$ is unattractive in wall-clock time since it does not benefit from the vectorized operation of MATLAB as larger choices of batch size and takes too many iterations.}. The experiment on Syn1 and Syn2 are similar but in Syn2 we put the constraint on a dictionary $U$, hence in Syn2 the projection operator has an additional cost of performing such orthogonal transform. In Syn1's wall-clock time plot we can see that SAGA with $b = 10$ has the fastest convergence among all the batch size choices, but in Syn2 it becomes the worst batch size choice for SAGA since it demands more iterations and hence more calls on the projection operator. In Syn3 we have a more expensive projection operator since our constraint is on the nuclear-norm of a matrix $X \in \mathbb{R}^{100 \times 100}$, and we can observe that the real convergence speed of SAGA with $b = 10$ become much slower than any other methods in terms of wall-clock time. In this scenario the full gradient method is much more competitive. However even here as the error reduces the sketched gradient methods exhibit a computational advantage.

\begin{table}[t]\label{Syn}
\caption{Synthetic data set settings. (*) U denotes the dense dictionary which is a orthogonal transform. (**) s denotes sparsity or rank of the ground truth}
\label{sample-table}
\vskip 0.15in
\begin{center}
\begin{small}
\begin{sc}
\begin{tabular}{lcccr}
\hline
\abovespace\belowspace
Data set & Size & (**)$s$ & $\frac{L}{\mu}$ & $\Phi$ \\
\hline
\abovespace
Syn1    & (100000, 100) & 10 & $10^7$ & I\\
Syn2    & (100000, 100) & 10 & $10^7$ & (*)U\\
\belowspace
Syn3 (low rank)      & (50000, 100) & 5 & $10^4$ & - \\
\hline
\end{tabular}
\end{sc}
\end{small}
\end{center}
\vskip -0.1in
\end{table}

\begin{table}[t]
\caption{Chosen data sets for Least-square regression, RFs: number of relevant features}
\label{sample-table}
\vskip 0.15in
\begin{center}
\begin{small}
\begin{sc}
\begin{tabular}{lcccr}
\hline
\abovespace\belowspace
Data set & Size & RFs & $\Phi$\\
\hline
\abovespace
year        & (500000, 90) &  90  & - \\
\belowspace
Magic04   & (19000, 10 + 40)  & 10 & I \\
\hline
\end{tabular}
\end{sc}
\end{small}
\end{center}
\vskip -0.1in
\end{table}

	\begin{figure} %\label{fig:3} %{\textwidth}
	\centering
		\includegraphics[width=9cm]{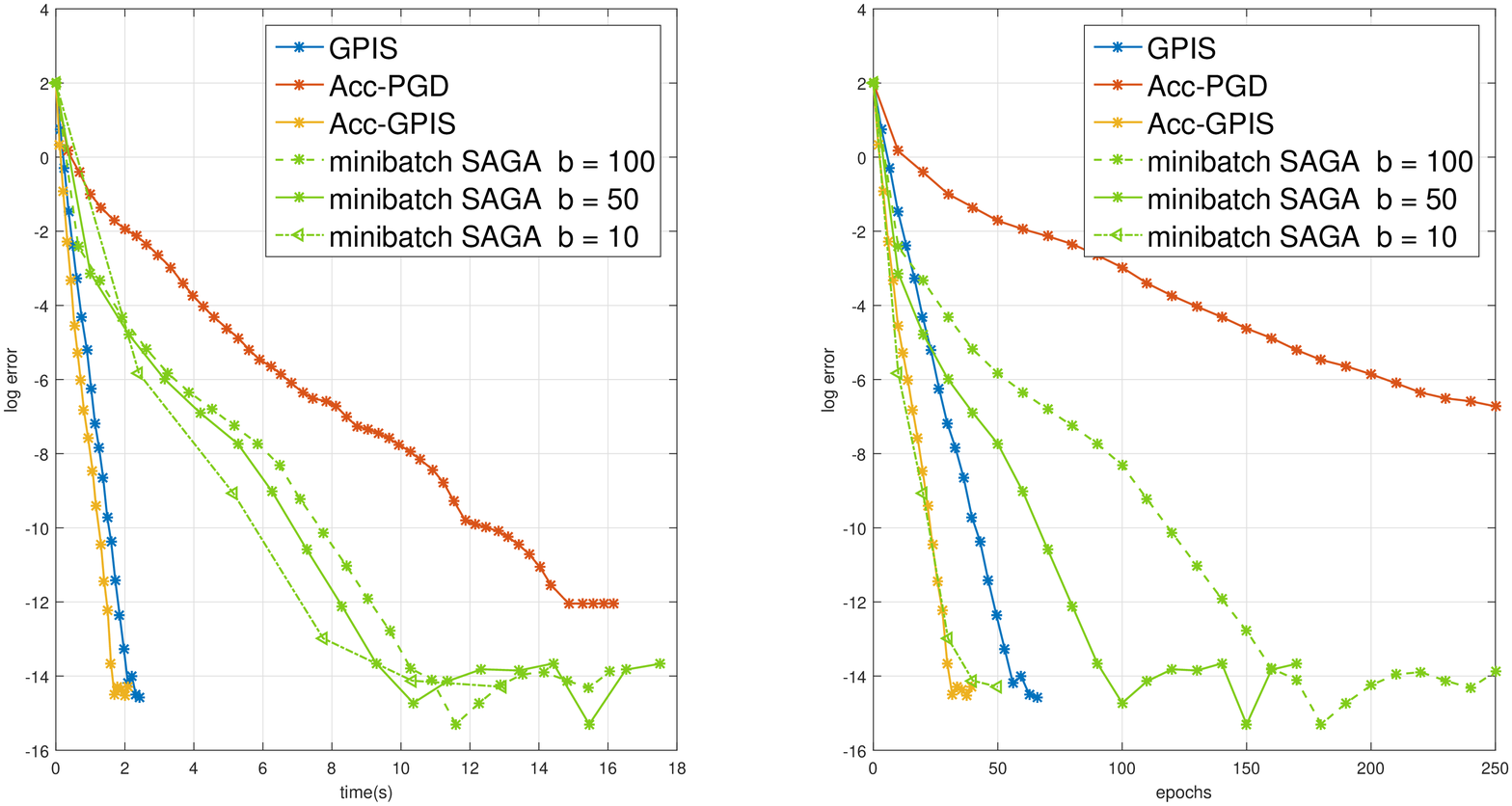}
		\includegraphics[width=9cm]{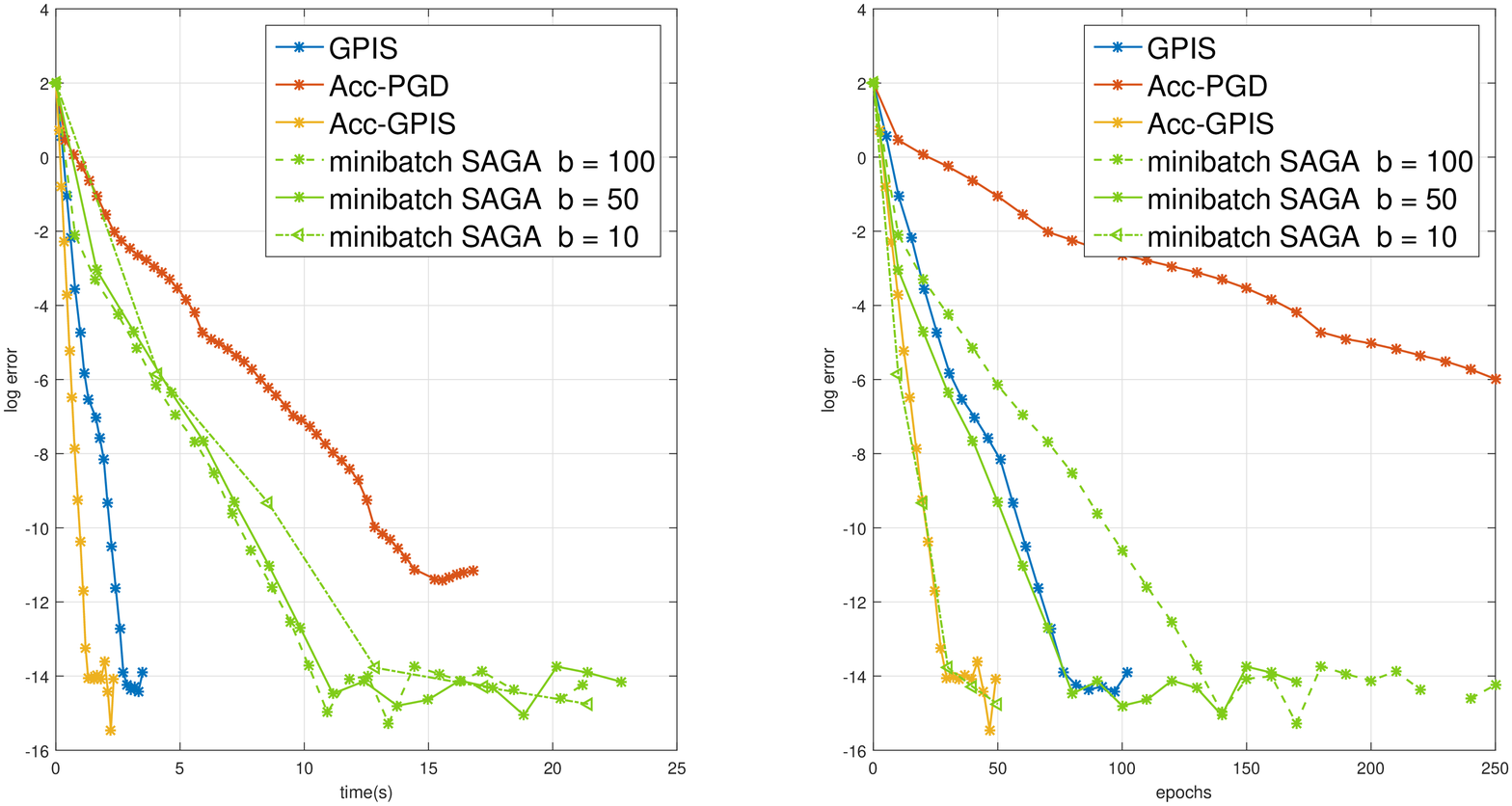}
		\includegraphics[width=9cm]{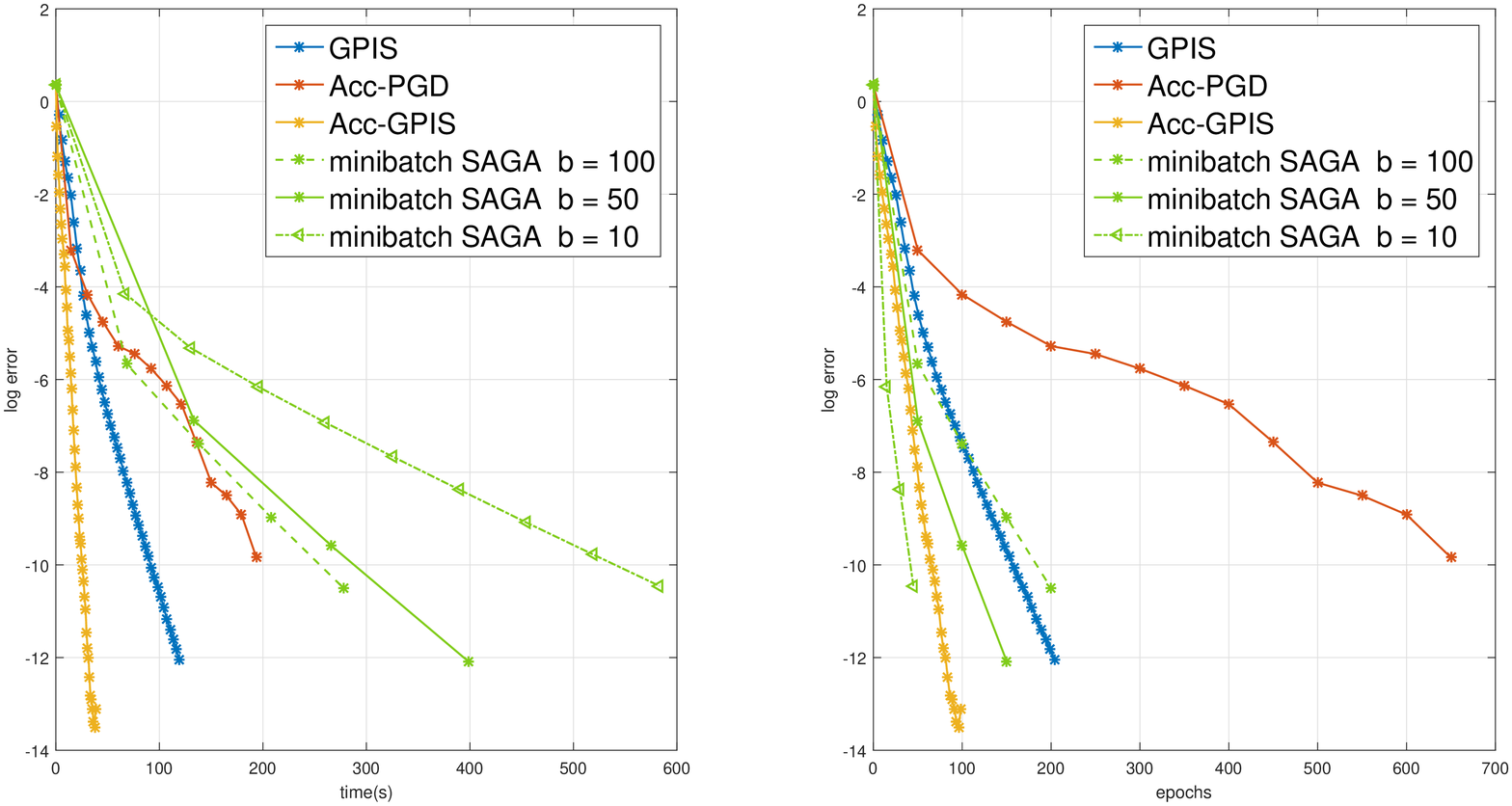}
		\includegraphics[width=9cm]{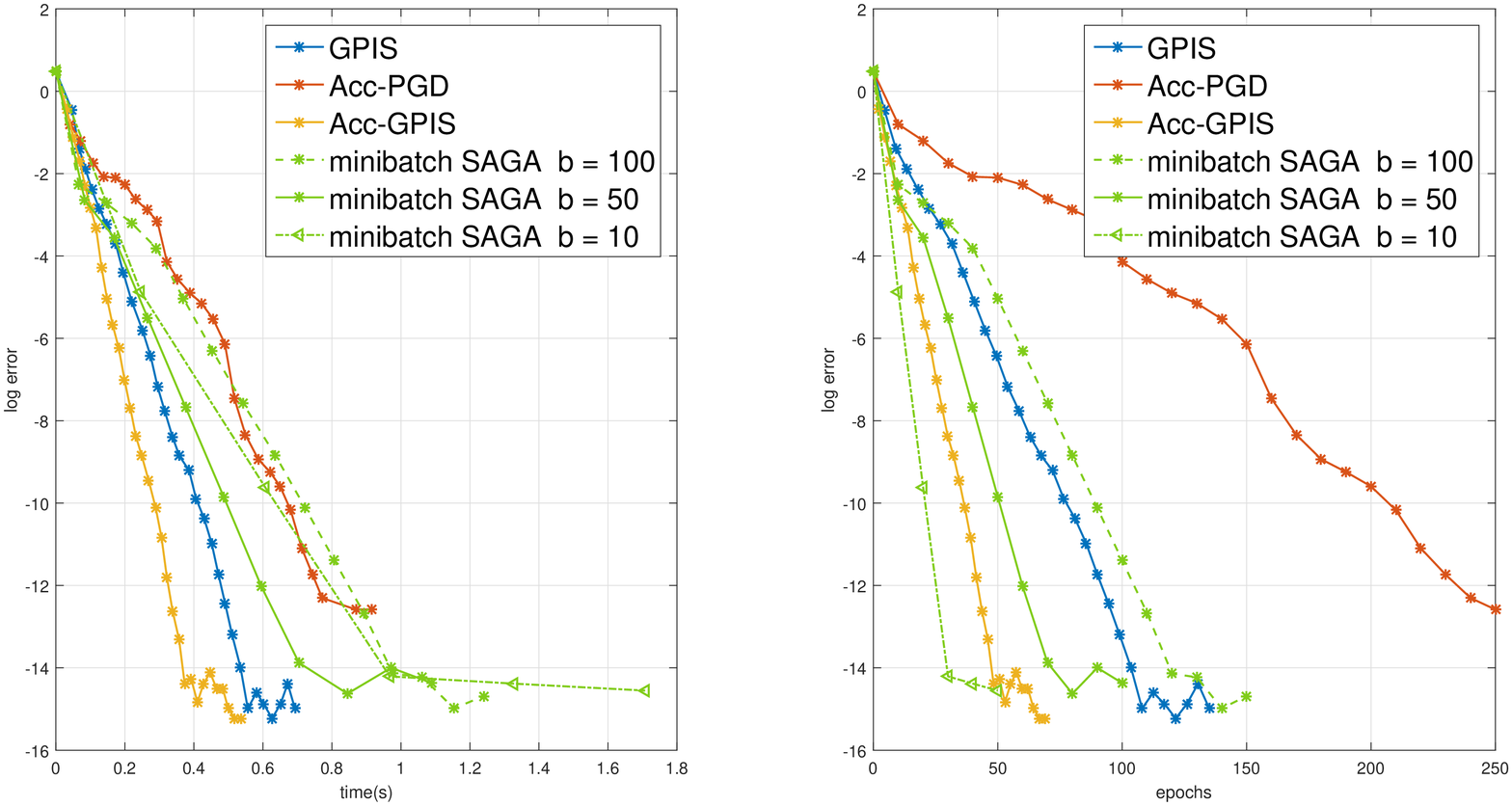}
	\caption{Experimental results on (from top to button) Syn1, Syn2, Syn3 and Magic04 data sets. The left column is for wall-clock time plots, while the right column is for epoch counts }
	\end{figure} %\vspace{-1cm} 

\section{Conclusions}


We propose two sketched gradient algorithms GPIS and Acc-GPIS for constrained Least-square regression tasks. We provide theoretical convergence analysis of the proposed algorithms for general sketching methods and high probability concentration bounds for the Gaussian sketches. The numerical experiments demonstrates that for dense large scale overdetermined data sets our sketched gradient methods performs very well compares to the stochastic gradient method (mini-batch) SAGA and the Accelerated full gradient method in terms of wall-clock time thanks to the benefits of sketched deterministic iterations, the efficient implementation of the Count-sketch and the use of aggressive line-search methods.
%In short, the strength of the sketched gradients is being able to make each of the iterations delicate hence is suitable for the applications where a fewer iteration (e.g. the prox/projection operator is expensive) and a large sketch size (this will be preferred when the machine is good at vectorized computation) or a fixed sketch is demanded.

\section*{Acknowledgements}
JT, MG and MD would like to acknowledge the support from H2020-MSCA-ITN Machine Sensing Training Network (MacSeNet), project 642685; EPSRC Compressed Quantitative MRI grant, number EP/M019802/1; and ERC Advanced grant, project 694888, C-SENSE, respectively. MD is also supported by a Royal Society Wolfson Research Merit Award. The authors also give thanks to the anonymous reviewers for insightful comments.

\nocite{langley00}

\newpage

\bibliography{example_paper}
\bibliographystyle{icml2017}

\section{Supplementary materials}

\subsection{the proof for Theorem 1}

\begin{proof}

At first we denote the underlying cost function of GPIS as $f_t(x)$:

for $ t = 0 $, we have the cost function of the classical sketch (CS):
\begin{equation}
    f_t(x) := \frac{1}{2}\|Sy-SAx\|^2_2,
\end{equation}
for $ t = 1 , 2 , ... , N $ we have the the cost function of Iterative Hessian Sketch (IHS):
\begin{equation}
    f_t(x) = \frac{1}{2}\|S^{t+1}A(x-x^t)\|_2^2-m x^TA^T(y-Ax^t),
\end{equation}
and then we denote  the optimal solution of $f_t$ constrained to set $\mathcal{K}$ as $x_\star^t$ and $\|r_{i+1}^t\|_2 
    = \|x_{i+1}^t-x_\star^t\|_2$ have:
\begin{equation}
\|r_{i+1}^t\|_2 
    = \|x_{i+1}^t-x_\star^t\|_2 
  =\|\mathcal{P}_\mathcal{K} (x_i^t-\eta\nabla f(x_i))-x_\star^t\|_2
\end{equation}
then we denote cone $\mathcal{C}_t$ to be the smallest close cone at $x_\star^t$ containing the set $\mathcal{K}-x_\star^t$, again because of the distance preservation of translation by Lemma 6.3 of \cite{2015_Oymak_Sharp}, we have:
\begin{equation}
    \begin{aligned}
        \|r_{i+1}^t\|_2 &=\|\mathcal{P}_{\mathcal{K}-x_\star^t} (x_i^t-\eta\nabla f(x_i)-x_\star^t)\|_2\\
        &=  \sup_{v \in \mathcal{C}_t \cap \mathcal{B}^d}\left\{v^T(x_i-x_\star^t - \mu \nabla f(x_i))\right\},
    \end{aligned}
\end{equation}
then because of the optimality condition on the constrained LS solution $x_\star^t$, we have:
\begin{equation}
    \begin{aligned}
        \|r_{i+1}^t\|_2 &=  \sup_{v \in \mathcal{C}_t \cap \mathcal{B}^d}\left\{v^T(x_i-x_\star^t - \eta \nabla f(x_i))\right\}\\
        & \leq \sup_{v \in \mathcal{C}_t \cap \mathcal{B}^d}\left\{v^T(x_i-x_\star^t - \eta \nabla f(x_i)) +  \eta  v^T\nabla f(x_\star^t)\right\}\\
        &= \sup_{v \in \mathcal{C}_t \cap \mathcal{B}^d}\left\{v^T(x_i-x_\star^t) -  \eta  v^T(\nabla f(x_i)-\nabla f(x_\star^t))\right\}\\
        &= \sup_{v \in \mathcal{C}_t \cap \mathcal{B}^d}\left\{v^T(I -  \eta  A^TS^TSA)r_i^t\right\}\\
        &\leq \sup_{u,v \in \mathcal{C}_t \cap \mathcal{B}^d}\left\{v^T(I -  \eta A^TS^TSA)u\right\}\|r_i^t\|_2\\
        & \leq \sup_{u,v \in \mathcal{B}^d}\left\{v^T(I -  \eta A^TS^TSA)u\right\}\|r_i^t\|_2,
    \end{aligned}
\end{equation}
We denote:
\begin{equation}
    \alpha_t= \sup_{u,v \in \mathcal{B}^d}v^T(I -  \eta A^TS^TSA)u,
\end{equation}
then by recursive subsitution we have:
\begin{equation}
    \|r_{i+1}^t\|_2 \leq \alpha_t^i\|r_0^t\|_2,
\end{equation}
and suppose we run GPIHS inner loop $k_t$ time, we have:
\begin{equation}
    \|r_{k_t+1}^t\|_2 \leq \left\{\alpha_t\right\}^{k_t}\|r_0^t\|_2,
\end{equation}
and we transfer it in terms of $A$-norm:
\begin{equation}
    \|r_{k_t+1}^t\|_A \leq \left\{\alpha_t\right\}^{k_t}\sqrt{\frac{L}{\mu}}\|r_0^t\|_A.    
\end{equation}
From the main theorems of the Classical sketch \cite{2015_Pilanci_Randomized} and Iterative Hessian Sketch \cite{2016_Pilanci_Iterative} we have following relationships:
\begin{equation}
       \|x_\star^0 -x^\star\|_A \leq 2\rho_0 \|A x^\star - y\|_2 = 2\rho_0 \|e\|_2,
\end{equation}
and,
\begin{equation}
       \|x_\star^t -x^\star\|_A \leq \rho_t \|x_0^{t} -x^\star\|_A.
\end{equation}
Then by triangle inequality we have:
\begin{equation}
      \|x_0^1 - x^\star\|_A \leq \|x_0^1 - x_\star^0\|_A + 2\rho_0 \|e\|_2,
\end{equation}
and,
\begin{equation}
       \|x_0^{t+1} -x^\star\|_A \leq \|x_0^{t+1} - x_\star^{t}\|_A + \rho_t \|x_0^{t} -x^\star\|_A.
\end{equation}
Then for $t = 0$ we can have:
\begin{equation}
    \begin{aligned}
     \|x_0^1 - x^\star\|_A &\leq \|x_0^{1} - x_\star^{0}\|_A + 2\rho_0 \|e\|_2\\
     &\leq \left\{\alpha_t\right\}^{k_t}\sqrt{\frac{L}{\mu}}\|x_0^{0} - x_\star^{0}\|_A + 2\rho_0 \|e\|_2,
    \end{aligned}
\end{equation}
for $t = 1, 2, ..., N$ we have:
\begin{equation}
\begin{aligned}
     &\|x_0^t - x^\star\|_A\\
     &\leq \|x_0^{t} - x_\star^{t - 1}\|_A + \rho_t \|x_{0}^{t-1} - x^\star\|_A\\
     &\leq \left\{\alpha_t\right\}^{k_t}\sqrt{\frac{L}{\mu}}\|x_{0}^{t-1} - x_\star^{t - 1}\|_A \\
     &+ \rho_t \|x_{0}^{t-1} - x^\star\|_A\\
     &\leq \left\{\left\{\alpha_t\right\}^{k_t}\left((1 + \rho_t)\sqrt{\frac{L}{\mu}}\right) + \rho_t\right\} \|x_{0}^{t-1} - x^\star\|_A,    
\end{aligned}
\end{equation}
The last inequality holds because:
\begin{equation}
    \begin{aligned}
    \|x_{0}^{t-1} - x_{f_{N-1}}^\star\|_A &\leq \|x_{0}^{t-1} - x^\star\|_A + \|x_\star^{t - 1} - x^\star\|_A    \\
    &\leq \left\{1 + \rho_t\right\} \|x_{0}^{t-1} - x^\star\|_A,
    \end{aligned}
\end{equation}
Then we denote:
\begin{equation}
    \rho_t^\star = \left\{\alpha_t\right\}^{k_t}\left((1 + \rho_t)\sqrt{\frac{L}{\mu}}\right) + \rho_t
\end{equation}
and do recursive substitution we can have:
    \begin{equation}
       \|x_0^t - x^\star\|_A 
       \leq  \left\{\prod_{t=1}^N\rho_t^\star\right\}\|x_0^1-x^\star\|_A.
    \end{equation}
hence we finish the proof of Theorem 1.
\end{proof}

\subsection{The proofs for Theorem 2 and 3}
\begin{proof}
From the theory of the Classical sketch and Iterative Hessian Sketch we have following relationships:
\begin{equation}
       \|x_\star^0 -x^\star\|_A \leq 2\rho_0 \|A x^\star - y\|_2 = 2\rho_0 \|e\|_2,
\end{equation}
and,
\begin{equation}
       \|x_\star^t -x^\star\|_A \leq \rho_t \|x_0^{t} -x^\star\|_A.
\end{equation}
Then by triangle inequality we have:
\begin{equation}
      \|x_0^1 - x^\star\|_A \leq \|x_0^1 - x_\star^0\|_A + 2\rho_0 \|e\|_2,
\end{equation}
and,
\begin{equation}
       \|x_0^{t+1} -x^\star\|_A \leq \|x_0^{t+1} - x_\star^{t}\|_A + \rho_t \|x_0^{t} -x^\star\|_A.
\end{equation}

The remaining task of this proof is just bound the term $\|x_0^{t+1} - x_\star^{t}\|_A$ for both GPIS and Acc-GPIS algorithm and then chain it. For all the sketched objective function $f_t(x)$ , $t = 0,1,...,N$, and any pair of vectors $x, x' \in \mathcal{K}$ we have:
\begin{equation}
       f_t(x) - f_t(x') - \langle \triangledown f_t(x') , x - x' \rangle = \|S^t A (x - x')\|_2^2
\end{equation}
If we set $x' = x_\star^t$, by first order optimality condition we immediately have:
\begin{equation}
\begin{aligned}
          &&f_t(x) - f_t(x_\star^t) 
         \geq \|S^t A (x- x_\star^t)\|_2^2 \\
         &=& \|S^t \frac{A (x- x_\star^t)}{\|A (x- x_\star^t)\|_2}\|A (x- x_\star^t)\|_2\|_2^2 \\
          &\geq& \left\{\inf_{v \in range(A) \cap \mathbb{S}^{n-1}} \|S^t v\|_2^2\right\} \|x- x_\star^t\|_A^2,
\end{aligned}
\end{equation}
so we have:
\begin{equation}
       \|x- x_\star^t\|_A \leq \frac{\sqrt{f_t(x) - f_t(x_\star^t)}}{\inf_{v \in range(A) \cap \mathbb{S}^{n-1}} \|S^t v\|_2},
\end{equation}

From the convergence theory in \cite{2009_Beck_Fast} which the authors in their Remark 2.1 have stated to hold for convex constrained sets, for GPIS inner iterates we have:
\begin{equation}
       f_t(x_k) - f_t(x_\star^t) \leq \frac{\beta LR\sup_{v \in range(A) \cap \mathbb{S}^{n-1}} \|S^t v\|_2^2 }{2k},
\end{equation}
and for Acc-GPIS inner loop we have:
\begin{equation}
       f_t(x_k) - f_t(x_\star^t) \leq \frac{2\beta LR\sup_{v \in range(A) \cap \mathbb{S}^{n-1}} \|S^t v\|_2^2 }{(k + 1)^2},
\end{equation}
hence for GPIS:
\begin{equation}
       \|x_0^{t+1} -x_\star^t\|_A \leq \sqrt{\frac{\beta L\sigma_t R}{2k}},
\end{equation}
for Acc-GPIS,
\begin{equation}
       \|x_0^{t+1} -x_\star^t\|_A \leq \sqrt{\frac{2\beta L\sigma_t R}{(k + 1)^2}},
\end{equation}
Then by simply towering the inequalities we shall obtain the desired results in Theorem 2 and 3.
\end{proof}

\subsection{The proofs for quantitative bounds of $\alpha_t$, $\rho_t$ and $\sigma_t$ for Gaussian sketches}

To prove the results in Proposition 1, 2 and 3 we need the following concentration lemmas as pillars:
\begin{lemma}\label{L1}
For any $g \in \mathcal{R}^d$, we have:
\begin{equation}
    \sup_{v \in \mathcal{C} \cap \mathcal{B}^{d}} v^Tg = \max\left\{ 0, 
            \sup_{u \in \mathcal{C} \cap  \mathbb{S}^{d-1}} u^Tg 
            \right\}
\end{equation}
\end{lemma}
\begin{proof}
By the definition of cone projection operator we have:
\begin{equation}
\sup_{v \in \mathcal{C} \cap \mathcal{B}^{d}} v^Tg = \|\mathcal{P_C}(g)\|_2 \geq 0
\end{equation}
if $\sup_{v \in \mathcal{C} \cap \mathcal{B}^{d}} v^Tg > 0$:
\begin{equation}
\sup_{v \in \mathcal{C} \cap \mathcal{B}^{d}} v^Tg = \sup_{v \in \mathcal{C} \cap \mathcal{B}^{d}} \|v\|_2\frac{v^Tg}{\|v\|_2} \leq \sup_{u \in \mathcal{C} \cap  \mathbb{S}^{d-1}} u^Tg,
\end{equation}
and meanwhile since $\mathcal{C} \cap  \mathbb{S}^{d-1} \in \mathcal{C} \cap \mathcal{B}^{d}$ we have:
\begin{equation}
    \sup_{v \in \mathcal{C} \cap \mathcal{B}^{d}} v^Tg \geq \sup_{u \in \mathcal{C} \cap  \mathbb{S}^{d-1}} u^Tg,
\end{equation}
hence we have:
\begin{equation}
    \sup_{v \in \mathcal{C} \cap \mathcal{B}^{d}} v^Tg = \sup_{u \in \mathcal{C} \cap  \mathbb{S}^{d-1}} u^Tg,
\end{equation}
\end{proof}

\begin{lemma}\label{L2}
If $\sup_{u, v \in \mathcal{C} \cap \mathcal{B}^{d}} v^TMu > 0$, we have:
\begin{equation}
    \sup_{u, v \in \mathcal{C} \cap \mathcal{B}^{d}} v^TMu = \sup_{u, v \in \mathcal{C} \cap  \mathbb{S}^{d-1}} v^TMu 
\end{equation}
\end{lemma}
\begin{proof}
Since $u, v \in \mathcal{C} \cap \mathcal{B}^{d}$, $\|u\|_2$ and $\|v\|_2$ are both less than or equal to 1, we can have the following upper bound:
\begin{eqnarray*}
    \sup_{u, v \in \mathcal{C} \cap \mathcal{B}^{d}} v^TMu &=& \sup_{u, v \in \mathcal{C} \cap \mathcal{B}^{d}} (\frac{v^TMu}{\|v\|_2\|u\|_2})\|v\|_2\|u\|_2 \\
    &\leq&  \sup_{u, v \in \mathcal{C} \cap  \mathbb{S}^{d-1}} v^TMu, 
\end{eqnarray*}
and meanwhile since $\mathcal{C} \cap  \mathbb{S}^{d-1} \in \mathcal{C} \cap \mathcal{B}^{d}$ we have:
\begin{equation}
    \sup_{u, v \in \mathcal{C} \cap \mathcal{B}^{d}} v^TMu \geq \sup_{u, v \in \mathcal{C} \cap  \mathbb{S}^{d-1}} v^TMu,
\end{equation}
hence we have:
\begin{equation}
    \sup_{u, v \in \mathcal{C} \cap \mathcal{B}^{d}} v^TMu = \sup_{u, v \in \mathcal{C} \cap  \mathbb{S}^{d-1}} v^TMu 
\end{equation}
\end{proof}

\begin{lemma}\label{L3}
If the entries of the sketching matrix $S$ is i.i.d drawn from Normal distribution and $v \in \mathcal{C}$, we have:
\begin{equation}
    \|SAv\|_2 \geq \sqrt{\mu}(b_m - \mathcal{W} - \theta) \|v\|_2,
\end{equation}
\begin{equation}
    \|SAv\|_2 \leq \sqrt{L}(b_m + \mathcal{W} + \theta) \|v\|_2,
\end{equation}
with probability at least $1-e^{-\frac{\theta^2}{2}}$. ($b_m=\sqrt{2}\frac{\Gamma(\frac{m+1}{2})}{\Gamma(\frac{m}{2})}\approx\sqrt{m}$, $\mathcal{W} := \mathcal{W} (A\mathcal{C}\cap  \mathbb{S}^{n-1})$)
\end{lemma}
\begin{proof}
This Lemma follows the result of the simplified form of the Gordon's Lemma [Lemma 6.7]\cite{2015_Oymak_Sharp}:
\begin{eqnarray*}
    \|SAv\|_2 &\geq& (b_m - \mathcal{W} (A\mathcal{C}\cap  \mathbb{S}^{n-1})- \theta)\|Av\|_2\\
    &\geq& \sqrt{\mu}(b_m - \mathcal{W} (A\mathcal{C}\cap  \mathbb{S}^{n-1})- \theta)\|v\|_2
\end{eqnarray*}
\begin{eqnarray*}
    \|SAv\|_2 &\leq& (b_m + \mathcal{W} (A\mathcal{C}\cap  \mathbb{S}^{n-1}) + \theta)\|Av\|_2\\
    &\leq& \sqrt{L}(b_m + \mathcal{W} (A\mathcal{C}\cap  \mathbb{S}^{n-1}) + \theta)\|v\|_2
\end{eqnarray*}
\end{proof}

\subsubsection{the proof for Proposition 1}

\begin{proof}
Let's mark out the feasible region of the step-size $\eta$:
\begin{eqnarray*}
&&\alpha(\eta,S^tA) \\
&=& \sup_{u,v \in \mathcal{B}^{d}} v^T(I-\eta A^TS^{T}SA)u \\
&\geq& \sup_{v \in \mathcal{B}^{d}} v^T(I-\eta A^TS^{T}SA)v \\
&=& \sup_{v \in \mathcal{B}^{d}} (\|v\|_2^2 - \eta \|SAv\|_2^2)\\
&\geq& \sup_{v \in \mathcal{B}^{d}} ((1 - \eta L(b_m + \sqrt{d} + \theta - \epsilon)^2)\|v\|_2^2),
\end{eqnarray*}
so if we choose a step size $\eta \leq \frac{1}{L(b_m + \sqrt{d} + \theta)^2}$  we can ensure that with probability $1-e^{-\frac{(\theta - \epsilon)^2}{2}}$($\epsilon > 0$) we have $\alpha(\eta,S^tA) > 0$ and the Lemma 2 become applicable:
\begin{eqnarray*}
&&\alpha(\eta,S^tA) \\
&=& \sup_{u,v \in \mathcal{B}^{d}} v^T(I-\eta A^TS^{T}SA)u \\
&=& \sup_{u,v \in  \mathbb{S}^{d-1}} v^T(I-\eta A^TS^{T}SA)u \\
&=& \sup_{u,v \in    \mathbb{S}^{d-1}}\frac{1}{4}[(u+v)^T(I-\eta A^TS^{T}SA)(u+v) \\
&&- (u-v)^T(I-\eta A^TS^{T}SA)(u-v)]\\
&=& \sup_{u,v \in    \mathbb{S}^{d-1}}\frac{1}{4}[\|u+v\|_2^2
-\eta\|SA(u+v)\|_2^2 \\
&&- \|u-v\|_2^2 + \eta \|SA(u-v)\|_2^2]\\
&\leq& \sup_{u,v \in    \mathbb{S}^{d-1}}\frac{1}{4}[(1
-\eta \mu(b_m - \sqrt{d} - \theta)^2)\|u+v\|_2^2\\
&&+ (\eta L(b_m + \sqrt{d} + \theta)^2 - 1)\|u-v\|_2^2]
\end{eqnarray*}
The last line of inquality holds with probability at least $1-2e^{-\frac{\theta^2}{2}}$ according to Lemma 3.
Then since we have set $\eta \leq \frac{1}{L(b_m + \sqrt{d}+ \theta + \epsilon)^2}$, and meanwhile notice the fact that $\|u + v\|_2^2 \leq 4$ we have:
\begin{eqnarray*}
&&\alpha(\eta,S^tA) \\
&\leq&  \sup_{u,v \in    \mathbb{S}^{d-1}}\frac{1}{4}(1
-\eta \mu(b_m - \sqrt{d} - \theta)^2\|u+v\|_2^2\\
&\leq& (1
-\eta \mu(b_m - \sqrt{d} - \theta)^2)
\end{eqnarray*}
If we chose $\eta = \frac{1}{L(b_m + \sqrt{d} + \theta)^2}$ we have:
\begin{equation}
    \alpha(\eta,S^tA) \leq \left(1-\frac{\mu}{L}\frac{(b_m-\sqrt{d}- \theta)^2}{(b_m+\sqrt{d}+ \theta )^2}\right),
\end{equation}
Then let $\epsilon \to 0$, we shall get the result shown in Proposition 1.
\end{proof}

\subsubsection{The proof for Proposition 2}
\begin{proof}
Recall that $\rho_t$ is defined as:
\begin{equation}
    \rho(S^t,A)=\frac{\sup_{v \in A\mathcal{C} \cap  \mathbb{S}^{n-1}} v^T ( \frac{1}{m}S^{t^T}S^t-I)z}{\inf_{v \in A\mathcal{C} \cap \mathbb{S}^{n-1}} \frac{1}{m}\|S^tv\|_2^2},
\end{equation}
we start by lower-bounding the denominator, by simplified Gordon's lemma [Lemma 6.7]\cite{2015_Oymak_Sharp} we directly have:
\begin{equation}
    \inf_{v \in A\mathcal{C} \cap \mathbb{S}^{n-1}} \frac{1}{m}\|Sv\|_2^2 \geq \frac{(b_m - \mathcal{W} - \theta)^2}{m},
\end{equation}
    with probability at least $(1-e^{-\frac{\theta^2}{2}})$.Then we move to the upper bound for the numerator:
\begin{equation}
    \begin{aligned}
    &v^T\left(\frac{S^{t^T}S^t}{m} - I\right)z \\
    &= \frac{1}{4}\{ (v + z)^T(\frac{S^{t^T}S^t}{m} - I)(v + z) \\
    &- (v - z)^T(\frac{S^{t^T}S^t}{m} - I)(v - z) \} \\
    &= \frac{1}{4}\{ \frac{1}{m}\|S^t(v + z)\|_2 -\|v + z\|_2\\
    & + \|v - z\|_2 - \frac{1}{m}\|S^t(v - z)\|_2 \},
    \end{aligned}
\end{equation}
and,
\begin{equation}
    \begin{aligned}
    \mathcal{W}(A\mathcal{C}\cap \mathbb{S}^{n-1} - z) &= \mathbb{E}_g(\sup_{v \in A\mathcal{C}\cap\mathbb{S}^{n-1}} g^T(v - z))\\
    &= \mathbb{E}_g(g^Tz + \sup_{v \in A\mathcal{C}\cap\mathbb{S}^{n-1}} v^Tg)\\
    & = \mathcal{W}(A\mathcal{C}\cap \mathbb{S}^{n-1})
    \end{aligned}
\end{equation}
hence we have the following by [Lemma 6.8]\cite{2015_Oymak_Sharp}:
\begin{equation}
    \begin{aligned}
    &v^T\left(\frac{S^{t^T}S^t}{m} - I\right)z \\
    &\leq \frac{1}{4}\left\{ \frac{1}{m}(b_m\|v + z\|_2 + \mathcal{W} + \theta)^2 - \|v + z\|_2^2\right\}\\
    & + \frac{1}{4}\left\{ \frac{1}{m}(b_m\|v - z\|_2 + \mathcal{W} + \theta)^2 - \|v - z\|_2^2\right\}\\
    & = \frac{1}{4}\left\{ (\frac{b_m^2}{m} - 1)\|v + z\|_2^2 + \frac{2b_m(\mathcal{W} + \theta)}{m} \|v + z\|_2 \right\}\\
    &+ \frac{1}{4}\left\{ (1 - \frac{b_m^2}{m})\|v - z\|_2^2 + \frac{2b_m(\mathcal{W} + \theta)}{m} \|v - z\|_2 \right\},
    \end{aligned}
\end{equation}
    with probability at least $(1-8e^{-\frac{\theta^2}{8}})$. Note that $\|v + z\|_2 + \|v - z\|_2 \leq 2\sqrt{2}$ and $\|v + z\|_2^2 + \|v - z\|_2^2 \leq 4$, we have:
\begin{equation}
    \begin{aligned}
    &v^T\left(\frac{S^{t^T}S^t}{m} - I\right)z \\
    &\leq \frac{2b_m(\mathcal{W} + \theta)}{m}\frac{\|v + z\|_2 + \|v - z\|_2}{4} + |\frac{b_m^2}{m} - 1|\\
    &\leq \frac{\sqrt{2}b_m(\mathcal{W} + \theta)}{m} + |\frac{b_m^2}{m} - 1|    
    \end{aligned}
\end{equation}
thus finishes the proof.
\end{proof}

\subsubsection{The proof for Proposition 3}
\begin{proof}
Recall that $\sigma_t$ is defined as:
\begin{equation}
\sigma(S^t, A)= \frac{\sup_{v \in range(A) \cap \mathbb{S}^{n-1}} \|S^t v\|_2^2}{\inf_{v \in  range(A) \cap \mathbb{S}^{n-1}} \|S^t v\|_2^2},
\end{equation}
by simply apply again the Gordon's lemma [Lemma 6.7]\cite{2015_Oymak_Sharp}, with $\mathcal{W}(A\mathbb{S}^{d-1}) \leq \sqrt{d}$, we with obtain the upper bound on the numerator:
\begin{equation}
    \sup_{v \in range(A) \cap \mathbb{S}^{n-1}} \|S^t v\|_2^2 \leq (b_m + \sqrt{d} + \theta)^2,
\end{equation}
and the lower bound:
\begin{equation}
    \inf_{v \in range(A) \cap \mathbb{S}^{n-1}} \|S^t v\|_2^2 \geq (b_m - \sqrt{d} - \theta)^2,
\end{equation}
both with probability at least $ 1 - e^{-\frac{\theta^2}{2}}$.
\end{proof}

\subsection{Details of the implementation of algorithms and numerical experiments}

For our GPIS and Acc-GPIS algorithms, we have several key points of implemenations:

\begin{itemize}
    \item {\bfseries Count sketch}
    
    As described in the main text.
    
    \item {\bfseries Line search}
    
    We implement the line-search scheme given by \cite{nesterov2007gradient} and is described by Algorithm 3 for GPIS and Acc-GPIS in our experiments with parameters $\gamma_u = 2$, and $\gamma_d = 2$. Such choice of line-search parameters simply means: when even we find the condition $f_t(\mathcal{P_K}(x_i - \eta \triangledown f_t(x_i))) \leq m_L$ does not hold, we shrink the step size by a factor of 2; and then at the beginning of each iteration, we increase the step size chosen at previous iteration by a factor of 2, then do backtracking again. Hence our methods are able to ensure we use an aggressive step size safely in each iteration. This is an important advantage of the sketched gradient method since we observe that for stochastic gradient such as SAGA a heuristic backtracking method similar to Algorithm 3 may work but it will demand a very small $\gamma_d$ (tends to 1) otherwise SAGA may go unstable, and an aggressive choice like our $\gamma_d = 2$ is unacceptable for SAGA. (Hence we suspect that SAGA is unlikely to be able to benefit computational gains from line-search as our method does.)
    
    \item {\bfseries Gradient restart for Acc-GPIS}
    
    \cite{o2015adaptive} has proposed two heuristic adaptive restart schemes - {\it gradient restart} and {\it function restart} for the accelerated gradient methods and have shown significant improvements without the need of the knowledge of the functional parameters $\mu$ and $L$. Such restart methods are directly applicable for the Acc-GPIS by nature due to its sketched deterministic iterations. Here we choose the {\it gradient restart} since it achieves comparable performance in practice as {\it function restart} but cost only $\mathcal{O}(d)$ operations.
    
\end{itemize}

\begin{figure} %\label{fig:3} %{\textwidth}
\centering
	\includegraphics[width=8.5cm]{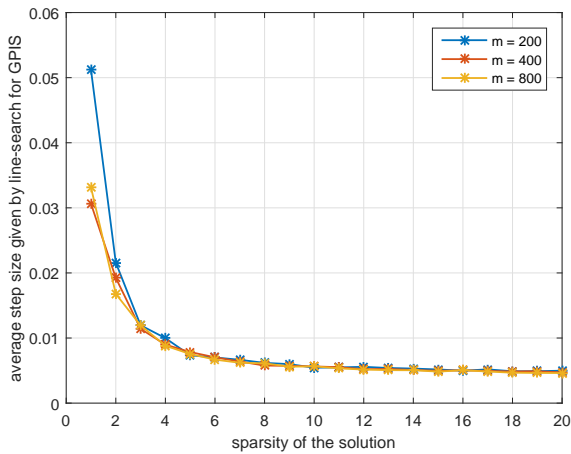}
\caption{Experimental results on the average choices of GPIS's step sizes given by line-search scheme \cite{nesterov2013gradient}}
\end{figure} %\vspace{-1cm} 

\subsubsection{Procedure to generate synthetic data sets}

The procedure we used to generate a constrained least-square problem sized $n$ by 100 with approximately $s$-sparse solution and a condition number $\kappa$  strictly follows:

1) Generate a random matrix $A$ sized $n$ by 100 with i.i.d entries drawn from $\mathcal{N}(0, 1)$.

2) Calculate $A$'s SVD: $A = U\Sigma V^T$ and replace the singular values $diag(\Sigma)_i$ by a sequence:

\begin{equation}
    diag(\Sigma)_{i} = \frac{diag(\Sigma)_{i - 1}}{\kappa^{\frac{1}{d}}}
\end{equation}

3) Generate the ''ground truth'' vector $x_{gt}$  sized 100 by 1 randomly with only $s$ non-zero entries in a orthongonal transformed domain $\Phi$, and calculate the $l_1$ norm of it ($r = \|\Phi x_{gt}\|_1$). Hence the constrained set can be described as $\mathcal{K} = \{x: \|\Phi x\|_1 \leq r\}$.

4) Generate a random error vector $w$ with i.i.d entries such that $\frac{\|Ax_{gt}\|_2}{\|w\|_2} = 10$.

5) Set $y = A x_{gt} + w$

\begin{table}[t]
\caption{Synthetic data set for step size experiment}
\label{sample-table}
\vskip 0.15in
\begin{center}
\begin{small}
\begin{sc}
\begin{tabular}{lcccr}
\hline
\abovespace\belowspace
Data set & Size & s & $\Phi$\\
\hline
\abovespace
Syn4        & (20000, 100) &  -  & I \\
\hline
\end{tabular}
\end{sc}
\end{small}
\end{center}
\vskip -0.1in
\end{table}

\subsubsection{Extra experiment for step size choice}

We explore the step size choices the GPIS algorithm produce through using the line-search scheme with respect to different sparsity level of the solution. The result we shown is the average of 50 random trials.

The result of the step-size simulation demonstrates that the step sizes chosen on average by the line-search scheme for the GPIS algorithm is actually related with the sparsity of the ground truth $x_{gt}$: at a regime when the $x_{{gt}}$ is sparse enough, the step size one can achieve goes up rapidly w.r.t the sparsity. While in our Proposition 2 we revealed that the outerloop of GPIS/Acc-GPIS can benefit from the constrained set, and here surprisingly we also find out numerically that the inner loop's can also benefit from the constrained set by aggressively choosing the large step sizes. Such a result echos the analysis of the PGD algorithm on constrained Least-squares with a Gaussian map $A$ \cite{2015_Oymak_Sharp}. Further experiments and theoretical analysis of such greedy step sizes for sketched gradients and full gradients on general maps is of great interest and will go beyond the state of the art analysis for convex optimization.

\end{document}